\documentclass[a4paper,12pt]{amsart}

\newtheorem{theo}{Theorem}[section]

\newtheorem{lemma}[theo]{Lemma}

\newtheorem{cor}[theo]{Corollary}

\newtheorem{ques}[theo]{Questions}

\newtheorem{rema}[theo]{Remark}
\newtheorem{remas}[theo]{Remarks}

\def \dem {\paragraph{ \em Proof.}}

\def \Romannumeral #1 {\expandafter\uppercase\expandafter {\romannumeral #1} }

\def \P {{\bf P}}

\def \calo {{\mathcal O}}

\def \Hom {{\rm {Hom}}}

\def \Z {{\bf Z}}
\def \Q {{\bf Q}}
\def \F {{\bf F}}

\def\smallsquare{\vbox{\hrule\hbox{\vrule height 1 ex\kern 1 ex\vrule}\hrule}}
\def\enddem{\hfill \smallsquare\vskip 3mm}

\def \lim {{\rm{lim\,}}}

\usepackage{amscd}
\usepackage{amssymb}
\usepackage{amsmath}
\usepackage{amsthm}
\usepackage{pb-diagram}

\everymath{\displaystyle}

\def \Gal {{\rm Gal}}

\DeclareFontFamily{U}{wncy}{}
\DeclareFontShape{U}{wncy}{m}{n}{%
   <5>wncyr5%
   <6>wncyr6%
   <7>wncyr7%
   <8>wncyr8%
   <9>wncyr9%
   <10>wncyr10%
   <11>wncyr10%
   <12>wncyr6%
   <14>wncyr7%
   <17>wncyr8%
   <20>wncyr10%
   <25>wncyr10}{}
\DeclareMathAlphabet{\cyrille}{U}{wncy}{m}{n}
\def\Sha{\cyrille X}

\usepackage{palatino}
\usepackage[all]{xy}

\usepackage{graphicx}
\usepackage[dvipsnames]{xcolor}

\makeatletter
\DeclareRobustCommand\widecheck[1]{{\mathpalette\@widecheck{#1}}}
\def\@widecheck#1#2{%
    \setbox\z@\hbox{\m@th$#1#2$}%
    \setbox\tw@\hbox{\m@th$#1%
       \widehat{%
          \vrule\@width\z@\@height\ht\z@
          \vrule\@height\z@\@width\wd\z@}$}%
    \dp\tw@-\ht\z@
    \@tempdima\ht\z@ \advance\@tempdima2\ht\tw@ \divide\@tempdima\thr@@
    \setbox\tw@\hbox{%
       \raise\@tempdima\hbox{\scalebox{1}[-1]{\lower\@tempdima\box
\tw@}}}%
    {\ooalign{\box\tw@ \cr \box\z@}}}
\makeatother

\title{On the defect in the generalized Grunwald--Wang problem}
\author{David Harari and Tam\'as Szamuely}
\address{Laboratoire de Math\'ematiques d'Orsay,
Universit\'e Paris-Saclay, 91405 Orsay, France}
\email{David.Harari@universite-paris-saclay.fr}
\address{Dipartimento di Matematica, Universit\`a di Pisa, Largo Bruno Pontecorvo 5, 56127 Pisa, Italy}
\email{tamas.szamuely@unipi.it}

\begin{document}

\maketitle

\noindent{\small{\sc Abstract.}  The classical Grunwald--Wang theorem asserts that, unless we are in the so-called special case, local cyclic Galois extensions at finitely many completions of a number field can be approximated by a global cyclic extension. In the special case the obstruction is measured by a group of order 2. It has been known for a long time that the Grunwald--Wang theorem extends to a very general context of valued fields. Therefore it is natural to ask whether in the special case the obstruction is always measured by a finite group and if so, is the order of this group bounded independently of the number of places under consideration. We show that the answer to both questions is negative in general, already for rational function fields and discrete valuations coming from points of the affine line. This has some interesting links to the arithmetic of function fields over $\Q$ or $\Q_p$.}

\section{Introduction}

Let $K$ be a field, $T$ a set of valuations of $K$, $M$ a finite \'etale commutative group scheme over $K$. Write $K_v$ for the completion of $K$ at $v$ and consider the product of restriction maps
$$
(r_v)_{v\in T}:\,H^1(K, M)\to\prod_{v\in T} H^1(K_v,M)
$$
for $v\in T$. In this setting equip each cohomology group with the discrete topology and
define $\overline{H^1(K,M)}$ as the closure of the image of $H^1(K,M)$ in the product. When $T$ is finite, this is just the image of $H^1(K,M)$, and when $T$ is infinite, we have
$$
\overline{H^1(K,M)}=\{ (\alpha_v)\in \prod_{v\in T} H^1(K_v,M)\mid {(\alpha_v)}_{v\in S}\in{\rm Im}{\left(r_v\right)}_{v\in S}\,\,\forall S\subset T, |S|<\infty \}
$$
by definition of the product topology.

\noindent We say that the {\em Grunwald-Wang theorem holds for $K$, $T$ and $M$} if $$\overline{H^1(K,M)}=\prod_{v\in T} H^1(K_v,M).$$

The classical case is of course when $K$ is a number field, $M=\Z/n\Z$ and $T$ is finite. As is well known, under these conditions the Grunwald--Wang theorem holds unless we are in the so-called {\em special case} (see e.g. \cite{nsw}, Theorem 9.2.8). The special case arises when $\Gal(K(\mu_{2^r})|K)$ is not cyclic, where $2^r$ is the largest power of 2 dividing $n$. The smallest counterexample (that of Wang \cite{wang}; see also \cite{artintate}, Chapter 10, Consequence of Theorem 1, p. 97) is for $K=\Q$, $M=\Z/8\Z$ and $T$ consisting only of the 2-adic valuation on $\Q$.

Apparently already Hasse had the insight that the theorem should hold more generally (for $M=\Z/n\Z$, $T$ finite and in the non-special case) in the above general context of valued fields. His approach was worked out in detail by Lorenz and Roquette in \cite{lr}. However, it was Saltman's method via generic Galois extensions \cite{saltman} and its reinterpretation by Colliot-Th\'el\`ene and Sansuc in \cite{ctsja} that opened the way for giving quick proofs of positive results on the Grunwald--Wang problem in a general setting. In fact, in (\cite{ctsja}, Theorem 8.4) it is proven that the Grunwald--Wang theorem holds in the above general context for finite $T$ and a large class of $M$ incuding $M=\Z/n\Z$ in the non-special case (see also Corollary \ref{cts} below).

In this note we focus on another, less studied aspect of the problem: that of the size of the quotient $\left(\prod_{v\in T} H^1(K_v,M)\right)/\overline{H^1(K,M)}$. In the classical case when $K$ is a number field the Poitou--Tate duality theorem implies that this quotient is always finite. For $M=\Z/n\Z$ one can even say more: the above quotient has order at most 2 (for the case of finite $T$ see e.g. Theorem 9.2.3 (ii) in \cite{nsw}; the general case is proven by similar arguments and passing to the limit). In light of this it is therefore natural to ask:

\begin{ques}\label{ques} ${}$
\begin{enumerate}
\item For general $K$, $T$ and $M$ is the quotient $\left(\prod_{v\in T} H^1(K_v,M)\right)/\overline{H^1(K,M)}$ always finite?

\item When the quotient is finite, is its order bounded independently of $T$?
\end{enumerate}
\end{ques}

A negative answer for question (2) implies a negative answer for question (1) for infinite $T$ but not necessarily for finite $T$. Therefore it is convenient to split question (1) in two, for finite and for infinite $T$. It should also be noted that for $M=\Z/n\Z$ the abelian group $\left(\prod_{v\in T} H^1(K_v,\Z/n\Z)\right)/\overline{H^1(K,\Z/n\Z)}$ is always of exponent 2 (see Corollary \ref{corgw} below), so in that case we are asking about dimensions of $\F_2$-vector spaces.

Though the answers to the above questions can be positive in several interesting cases besides that of number fields, we shall see that in general the answer is no for both questions, even when $K$ is a rational function field over a field $k$ of characteristic 0 and $T$ is a set of discrete valuations coming from closed points of $\P^1_k$. Regarding the first question, we shall show:

\begin{theo}\label{counter1}
There exists a field $k$ of characteristic zero such that for every finite
set $T$ of discrete valuations of $K:=k(t)$ coming from rational points of ${\bf P}^1_k$
the diagonal map
$$ H^1(K,\Z/8\Z) \to \prod_{v \in T} H^1(K_v,\Z/8\Z)$$
has infinite cokernel, provided $T$ has at least 2 elements.
\end{theo}

In the above theorem the field $k$ is of infinite transcendence degree over $\Q$. In fact, we shall see that for $k$ a finitely generated extension of $\Q$ and $T$ as in the theorem the quotient $\left(\prod_{v\in T} H^1(K_v,M)\right)/\overline{H^1(K,M)}$ is always finite for a finite \'etale commutative $k$-group scheme $M$. However, Question \ref{ques} (2) has a negative answer even for $K=\Q(t)$ or $\Q_2(t)$.

\begin{theo}\label{counter2} Let $k$ be $\Q$ or $\Q_2$, and let $T$ be a finite
set of discrete valuations of $K:=k(t)$ coming from closed points of ${\bf P}^1_k$.
The cokernel of the diagonal map
$$ H^1(K,\Z/8\Z) \to \prod_{v \in T} H^1(K_v,\Z/8\Z)$$
can have arbitrarily large dimension over $\F_2$ as $T$ varies.
\end{theo}

In fact, in the case $k=\Q_2$ we shall find examples where the $v\in T$ come from rational points.

\begin{cor} \label{corcount2}

Let $k=\Q$ (resp. $k=\Q_2$).
If $T$ is an infinite set of discrete valuations of $K$ coming
from closed points of ${\bf P}^1_\Q$ (resp. from rational points
of ${\bf P}^1_{\Q_2}$), the quotient $\left(\prod_{v\in T} H^1(K_v,\Z/8\Z)\right)/\overline{H^1(K, \Z/8\Z)}$ can be infinite.
\end{cor}

{As a by-product of the above result, we also obtain:

\begin{cor}\label{cormu}
In the case $k=\Q_2$ the subgroup $\Sha^2_{\omega}(K,\mu_8 ^{\otimes 2})\subset H^2(K,\mu_8 ^{\otimes 2})$ consisting of elements with trivial image in
$H^2(K_v,\mu_8 ^{\otimes 2})$ for all but finitely many $v \in {\bf P}^1_k$ is infinite.
\end{cor}

The interest in the groups $\Sha^2_{\omega}(K,\mu_m ^{\otimes 2})$ stems from the fact that they control weak approximation on quotients of semisimple simply connected $K$-groups by finite abelian constant subgroups. Whether they can be nontrivial was asked by M. L. Nguyen in (\cite{manlinh},
Question 3 at the end of Section 3.4).}

We shall prove the above statements in more general form in Sections \ref{sec1} and \ref{sec2}.

\section{Preliminaries}

Our main tool in studying Questions \ref{ques} is Saltman's method in the adaptation of Colliot-Th\'el\`ene and Sansuc \cite{ctsja}. If $M$ is a group of multiplicative type over a field $K$, (\cite{ctsja}, Proposition 1.3) gives an exact sequence
\begin{equation}\label{fres}
1\to M\to F\to P\to 1
\end{equation}
with $F$ a flasque and $P$ a quasi-trivial torus, called a {\em flasque resolution of $M$.} Recall that a torus $F$ is {\em flasque} if its cocharacter
 group $\widecheck F$ satisfies $H^1(H, \widecheck F)=0$ for
 every subgroup $H$ of the Galois group of a finite Galois field
 extension that splits $F$, and a torus $P$ is {\em quasitrivial} if  its character module is a permutation module.
Here if $M$ splits (i.e. becomes diagonalizable) over a finite Galois extension $L|K$, then we may find $F$ and $P$ that split over $L$ as well.

The following lemma, which will serve several times, is a slight variant of an argument in the proof of (\cite{ctsja}, Proposition 8.4).

\begin{lemma}\label{keylem}
In the above situation let $T$ be a finite set of discrete valuations of $K$ and let $K_v$ be the completion of $K$ with respect to $v\in T$. We have a canonical isomorphism
$$
{\rm Coker}(H^1(K, M)\to\prod_{v\in T} H^1(K_v,M))\cong {\rm Coker}(H^1(K, F)\to \prod_{v\in T} H^1(K_v,F)).
$$
\end{lemma}

\begin{proof}
Consider the commutative exact diagram
$$
\begin{CD}
F(K) @>>> P(K) @>>> H^1(K,M) @>>> H^1(K, F) @>>> 0 \\
@VVV @VVV @VVV @VVV  \\
\prod_{v\in T} F(K_v) @>>> \prod_{v\in T} P(K_v) @>>> \prod_{v\in T} H^1(K_v,M) @>>> \prod_{v\in T} H^1(K_v, F) @>>> 0
\end{CD}
$$
where the zeros on the right come from the fact that $P$ has trivial first cohomology over any extension of $K$ by Hilbert's Theorem 90 and Shapiro's lemma. We endow the groups $F(K_v)$ and $P(K_v)$ with the $v$-adic topology.
The first map in the lower row has open image by the implicit function theorem and the second vertical map has dense image by the weak approximation lemma. The lemma follows by a diagram chase.
\end{proof}

The following obvious consequence will be crucial for our constructions.

\begin{cor}\label{corkeylem}
If ${\rm Coker}(H^1(K, M)\to \prod_{v\in T} H^1(K_v,M))\neq 0$, then $H^1(K_v, F)\neq 0$ for some $v\in T$.
\end{cor}

The lemma also yields the result of Colliot-Th\'el\`ene and Sansuc
(\cite{ctsja}, Theorem~8.4) mentioned in the introduction.

\begin{cor}[\cite{ctsja}, Proposition 8.4 (ii)]\label{cts}
The Grunwald--Wang theorem holds when $M$ is a group of multiplicative type that becomes diagonalizable over a metacyclic extension of $K$.
\end{cor}

Recall that a finite group is called metacyclic if its $p$-Sylow subgroups are cyclic.

\begin{proof} Since $M$ is split by a metacyclic extension, so is $F$. But then by a theorem of Endo and Miyata (\cite{endomiy}, Theorem 1.5; see also \cite{requ},
Proposition 2, p. 184) the torus
$F$ is a direct factor in a quasi-trivial torus, hence $H^1(K,F)=0$. The same argument gives $H^1(K_v,F)=0$ for all $v$. Now apply the lemma.
\end{proof}

In fact, with a little more work, Colliot-Th\'el\`ene and Sansuc get the vanishing of the groups $H^1(K_v,F)$ under the weaker assumption that $M_{K_v}$ is split by a metacyclic extension, which suffices for the Grunwald-Wang property. For us what is important from this discussion is the following corollary:

\begin{cor}\label{corgw} Let $n$ be an integer not divisible by the characteristic of $K$.
If $\sqrt{-1}\in K$, then the Grunwald--Wang theorem holds for $M=\Z/n\Z$. In general the abelian group ${\rm Coker}(H^1(K, \Z/n\Z)\to \prod_{v\in T} H^1(K_v,\Z/n\Z))$ is  of exponent 2.
\end{cor}

\begin{proof}
For the first statement, note that $\Z/n\Z$ as a group of multiplicative type is split by the Galois extension $K(\mu_n)|K$. By (\cite{cassfro}, beginning of Chapter 3, p. 85),
the $p$-Sylow subgroups of $\Gal(K(\mu_n)|K)$ are always cyclic for $p\neq 2$ and also for $p=2$ when $\sqrt{-1}\in K$ (`we are not in the special case'), so the previous corollary applies. For the second statement  we may assume $T$ finite and denote by $T'$ the set of places of $K(\sqrt{-1})$ dividing those in $T$. Given $(\alpha_v)\in \prod_{v\in T} H^1(K_v,\Z/n\Z)$, its restriction ${\rm Res}(\alpha_v)\in \prod_{w\in T'} H^1(K(\sqrt{-1})_w,\Z/n\Z)$ is the diagonal image of some $\alpha'\in H^1(K(\sqrt{-1}),\Z/n\Z)$ by the first statement. By functoriality of corestriction $2(\alpha_v)=({\rm Cor}\circ{\rm Res})(\alpha_v)$ is then the image of ${\rm Cor}(\alpha')\in H^1(K,\Z/n\Z)$.
\end{proof}

{

\begin{rema}\rm Though we shall not need it, we remark for the sake of completeness that when $K$ has characteristic $p>0$, the Grunwald--Wang theorem holds for finite \'etale $p$-torsion $M$. This is proven by the same arguments as in the case of global function fields (see \cite{nsw}, Theorem 9.2.5).

\end{rema}}

\section{The general construction}\label{sec1}

In this section we assume $k$ is a field of characteristic 0, and let $M$ be a finite \'etale commutative $k$-group scheme. As in (\ref{fres}), we choose a flasque resolution
$$
1\to M\to F\to P\to 1
$$
over $k$. Recall from (\cite{ctsja}, Lemma 0.6, exact sequence (0.6.2)
joint with paragraph 0.2)
that though the flasque torus $F$ in the above resolution is not unique, it is unique up to multiplication by a quasi-trivial torus. Since $H^1$ of a quasi-trivial torus is trivial, the group $H^1(k, F)$ depends only on $M$ and not on the choice of flasque resolution. Therefore the assumptions in the following theorem are unambiguous.

\begin{theo} \label{gencounter}

 With notation as above, set $K:=k(t)$ and
let moreover $T$ be a set of discrete valuations of $K$ coming from rational points of the projective line ${\bf P}^1_k$.

The quotient
$\left(\prod_{v \in T} H^1(K_v,M)\right)/\overline{H^1(K,M)}$
is infinite under either of the following conditions:
\begin{enumerate}
\item The group $H^1(k, F)$ is infinite and $T$ has at least 2 elements.
\item The group $H^1(k, F)$ is nontrivial and $T$ is infinite.
\end{enumerate}
\end{theo}

Recall that the group $H^1(k, F)$ is known to be finite for many types of fields $k$ (see the next section) but may be infinite (see Section 8 of \cite{6auth} for several examples, one of which we shall use below).

We thank Jean-Louis Colliot-Th\'el\`ene for simplifying our already quite simple proof.

\begin{proof} Assume $T$ is finite. By Lemma \ref{keylem} we then have

\begin{align*}\left(\prod_{v \in T} H^1(K_v,M)\right)/&\overline{H^1(K,M)} \cong \left(\prod_{v \in T} H^1(K_v,F)\right)/ \overline{H^1(K, F)}\cong \\ &\cong {\rm Coker}\,(H^1(K,F)/H^1(k,F)\to \left(\prod_{v \in T} H^1(K_v,F)\right) /H^1(k,F)).
\end{align*}

\noindent Now recall that the natural map $H^1(k,F)\to H^1(K,F)$ is an isomorphism; this follows from the case $n=1$ of (\cite{ctsja}, Corollary 2.6) by passing to the direct limit over Zariski open sets of the affine line over $k$.   Thus we are left with the quotient $(\prod_{v \in T} H^1(K_v,F)) /H^1(k,F)$.
Let $\calo_v$ be the ring of integers of $K_v$; it has residue field $k$ as $v$ comes from a rational point. As $\calo_v$ is complete and
$F$ is flasque, we have isomorphisms
\begin{equation}\label{eqgencounter}
H^1(k,F) \simeq H^1(\calo_v,F) \simeq H^1(K_v,F).
\end{equation} Here the first isomorphism holds by (\cite{milne}, Remark~III.3.11 (a)), injectivity of the map $H^1(\calo_v,F) \to H^1(K_v,F)$ by the main result of \cite{nis} and surjectivity by
(\cite{ctsja}, Theorem~2.2 (i)). Thus if $H^1(k, F)$ is infinite, the
assertion follows in case (1). If we only know that $H^1(k, F)$ is nontrivial, then the same argument shows that the size of the quotient is unbounded as we increase $T$, whence the assertion in case (2).\end{proof}

As a first consequence, we derive Theorem \ref{counter1}
 from case (1) of Theorem \ref{gencounter} above.

\begin{proof}[Proof of Theorem \ref{counter1}] We want to
show that there exists a field $k$ of characteristic zero
 such that for every finite
set $T$ of discrete valuations of $K:=k(t)$ coming
 from rational points of ${\bf P}^1_k$
the diagonal map
$$ H^1(K,\Z/8\Z) \to \prod_{v \in T} H^1(K_v,\Z/8\Z)$$
has infinite cokernel (provided $T$ has at least 2 elements).

We start with a field $k_0$ such that in some flasque resolution
$$1 \to \Z/8\Z \to F_0 \to P_0 \to 1$$ over $k_0$
the flasque torus $F_0$ satisfies $H^1(k_0,F_0) \neq 0$. For instance
$k_0=\Q_2$ has this property thanks to Wang's counterexample and Corollary \ref{corkeylem}. Now
by (\cite{6auth}, Example~8.4)
there exists an (infinitely generated) field extension $k$ of $k_0$ such that
$H^1(k,F)$ is infinite, where $F$ is the base change of $F_0$ to $k$. Moreover, if $P$ is the base change of $P_0$ to $k$, then
$0 \to \Z/8\Z \to F \to P \to 0$
is a flasque resolution over $k$ by (\cite{ctsja}, Proposition~1.4). Now apply Theorem~\ref{gencounter} (1).
\end{proof}

\section{Arithmetic base fields}\label{sec2}

We keep notation and assumptions from the previous section. In particular, $K=k(t)$ is a rational function field but from now on we assume moreover that the base field $k$ is a finite extension of $\Q_p$ or a finitely generated extension of $\Q$. In this case the groups $H^1(K_v, F)$ are known to be finite for a flasque $k$-torus $F$. Indeed, the field $K_v$ is a complete discrete valuation field of equal characteristic zero, hence it is a formal Laurent series field over its residue field $\kappa$  (see e.g. \cite{corpslocaux}, Section II.4, Theorem 2). Therefore by (\cite{ctgp}, Theorem 3.2) the finiteness of $H^1(K_v, F)$ follows from that of $H^1(\kappa,F)$, which in turn holds by (\cite{cogal}, section III.4, Theorem 4)
in the $p$-adic case and by (\cite{ctsja}, Corollary 2.8) in the global case.

Applying Lemma \ref{keylem}, we therefore see that the groups $\left(\prod_{v \in T} H^1(K_v,M)\right)/\overline{H^1(K,M)}$ are finite for a finite set of discrete valuations $T$ and an arbitrary $k$-group $M$ of multiplicative type. However, for infinite $T$ this may not be the case. We start with the local case:

\begin{cor} \label{infiniq2}
For $K=\Q_2(t)$,
the groups $\left(\prod_{v \in T} H^1(K_v,\Z/8\Z)\right)/\overline{H^1(K,\Z/8\Z)}$ are infinite if $T$ is an infinite set of rational points of
${\bf P}^1_k$.
\end{cor}

\begin{proof}
Again by Wang's counterexample and Corollary \ref{corkeylem} the flasque torus $F$ in a flasque resolution of $\Z/8\Z$ over $\Q$ satisfies $H^1(\Q_2, F)\neq 0$, so we conclude by Theorem \ref{gencounter} (2).
\end{proof}

This proves Theorem \ref{counter2} and Corollary~\ref{corcount2}
in the case $k=\Q_2$. { Note that when $K=\Q_p(t)$ with $p>2$, the quotient in the corollary is trivial (`we are not in the special case'). Indeed, the
cyclotomic extension $\Q_p(\mu_8)|\Q_p$ is unramified, hence cyclic by (\cite{corpslocaux}, Section IV.4, Proposition 16) and therefore $K(\mu_8)|K$ is a cyclic extension as well.

\smallskip

We can also derive Corollary \ref{cormu}. Recall the context: let $T$
be a finite set of closed points of ${\bf P}^1_k$. For $i>0$ and a finite \'etale commutative $k$-group scheme $M$ we denote by $\Sha^i_{\omega}(K,M)$ (resp. $\Sha^i(K,M)$, resp. $\Sha^i_T(K,M)$)
the subgroup of classes in $H^i(K,M)$ whose image is trivial in $H^i(K_v,M)$ for all but finitely many
(resp. for all, resp. for all outside $T$) closed points $v$ of ${\bf P}^1_k$. Let us restate the corollary:

 \begin{cor}
For $K=\Q_2(t)$ the group $\Sha^2_{\omega}(K,\mu_8 ^{\otimes 2})$ is infinite.
\end{cor}

\begin{proof} For $M$ of order $n$ as above we have an exact sequence
\begin{equation}\label{shaex}0 \to \overline{{H^1(K,M)}} \to \prod_{v} H^1(K_v,M) \to
\Sha^2_{\omega}(K,\Hom(M,\mu_n ^{\otimes 2}))^D \to \Sha^2(K,\Hom(M,\mu_n ^{\otimes 2}))^D \to 0,\end{equation}
where the superscript ${}^D$ denotes $\Q/\Z$-linear dual and in the product $v$ runs over all closed points of ${\bf P}^1_k$. This follows after passing to the direct limit over all finite sets $T$ of $v$ as above from the exact sequence
$$H^1(K,M) \to \prod_{v \in T} H^1(K_v,M) \to
\Sha^2_{T}(K,\Hom(M,\mu_n ^{\otimes 2}))^D \to \Sha^2(K,\Hom(M,\mu_n ^{\otimes 2}))^D \to 0$$
proven in more general form in (\cite{izq}, Lemme 2.2). In exact sequence (\ref{shaex}) the map $\overline{{H^1(K,M)}} \to \prod_{v} H^1(K_v,M)$ has infinite cokernel by Corollary~\ref{infiniq2} and the group $\Sha^2(K,\Hom(M,\mu_n ^{\otimes 2}))$ is finite and dual to $\Sha^2(K,M)$ by (\cite{pcwa}, Theorem~1.4). This implies the statement.
\end{proof}

\begin{remas}\rm ${}$
\begin{enumerate}
\item The previous two corollaries hold not just for 8 but for any $2^r$ with $r\geq 3$.
\item Lucas Lagarde informs us that he can construct infinitely many explicit \mbox{classes} in the group $\Sha^2_{\omega}(K,\mu_8^{\otimes 2})$ considered above; see his forthcoming paper \cite{lagarde} for details.
\item Consider still the case $K=\Q_2(t)$ and the flasque torus $F$ in the proof of Corollary  \ref{infiniq2}.   By (\cite{pcwa}, Theorem 4.3 (b)) we have an exact sequence
$$
0\to\overline{H^1(K,F)}\to\prod_{v} H^1(K_v,F)\to \Sha^2_\omega(K,F')^D\to \Sha^1(K,F)\to 0
$$
analogous to (\ref{shaex}), where $F'$ is the dual torus to $F$ (i.e. the torus whose character group is the cocharacter group of $F$). Again the last group is known to be finite (\cite{pcwa}, Theorem 2.4), so Corollary \ref{infiniq2} implies that $\Sha^2_\omega(K,F')$ must be infinite. This gives an example of a torus with infinite $\Sha^2_\omega$ different from that of (\cite{pcwa}, Prop. 4.5).
\end{enumerate}
\end{remas}
}

We finally turn to the proof of Theorem \ref{counter2} in the case of $\Q(t)$. We need the following lemma which must be well known:

\begin{lemma} \label{closedpoint}
Let $k$ be an infinite field, and let $L$ be a finite separable extension of $k$.
There exist infinitely many closed points of ${\bf A}_k ^1$ with
residue field isomorphic to $L$.
\end{lemma}

\dem As the case $L=k$ is obvious, we can assume $[L:k] \geq 2$. Since $L|k$ is separable, there are only finitely many intermediate fields $k\subsetneq k'\subsetneq L$ (e.g. by Galois theory applied to the Galois closure of $L$). The Weil restriction ${\bf R}_{L|k} {\bf A}_L^1$ is isomorphic to ${[L:k]}$-dimensional affine space, and the ${\bf R}_{k'|k} {\bf A}_{k'} ^1$ with $k'$  as above give rise to a finite collection of proper closed subvarieties of smaller dimension. As $k$ is infinite, their open complement
has infinitely many $k$-points, which then correspond to closed points of ${\bf A}^1_k$ with residue field $L$.
\enddem

\begin{proof}[Proof of Theorem \ref{counter2} for $K=\Q(t)$]
Assume now that $M=\Z/8\Z$. As in the previous section, Wang's counterexample and Corollary \ref{corkeylem} imply that $H^1(\Q_2,F)\neq 0$, where $F$ is the flasque torus in a flasque resolution of $\Z/8\Z$ over $\Q$.
But by (\cite{adtcrelle}, Lemma 2.7) we have an isomorphism $H^1(\Q_2,F)\cong H^1(\Q_2^h,F)$,
where $\Q_2^h$ denotes the henselization of $\Q$ at $2$ which is an
algebraic extension of $\Q$.
Thus there is a finite extension $L|\Q$ contained in $\Q_2^h$
with $H^1(L,F)\neq 0$. By Lemma~\ref{closedpoint} above there are infinitely many closed points of ${\bf P}^1_\Q$ with residue field $L$. Let $T$ be
a finite set of discrete valuations coming from these points. Using Lemma \ref{keylem} and (\cite{ctsja}, Corollary 2.6) as in the proof of Theorem \ref{gencounter}, we obtain a chain of isomorphisms

$$(\prod_{v \in T} H^1(K_v,M))/\overline{H^1(K,M)}\cong (\prod_{v \in T} H^1(K_v,F))/
\overline{H^1(K, F)}\cong (\prod_{v \in T} H^1(K_v,F))/H^1(\Q, F).$$

\noindent Next, as in formula (\ref{eqgencounter}) in the proof of Theorem \ref{gencounter}, we get  ${H^1(K_v,F)\cong H^1(L, F)\neq 0}$ for all $v\in T$. Since $H^1(\Q, F)$ is finite (again by \cite{ctsja}, Corollary 2.8), the size of the quotient $(\prod_{v \in T} H^1(K_v,F))/H^1(\Q, F)$ can be arbitrarily large as we add more and more $v$ with residue field $L$ to $T$.
\end{proof}

\begin{rema}\rm
A direct generalization of the above argument shows the following statement. Suppose $k$ is a field of finite type over its prime field and $M$ is a finite \'etale commutative $k$-group scheme for which the Grunwald--Wang property fails for some finite set of discrete valuations of $k$. Then the same holds for $M$ over $k(t)$ and a suitably chosen finite set of places $T$. Moreover,  the cokernel can be arbitrarily large when $T$ varies.
\end{rema}


\begin{thebibliography}{99}

\bibitem{artintate} E. Artin, J. Tate, {\em Class Field Theory}, W. A. Benjamin, Reading, 1967.

\bibitem{cassfro} J.W.S. Cassels, A. Fr\"ohlich, {\em Algebraic number theory}, Academic Press, London and New York, 1967.

\bibitem{ctgp} J.-L. Colliot-Th\'el\`ene, P. Gille, R. Parimala, {
Arithmetic of linear algebraic groups over two-dimensional geometric fields},
{\em Duke Math. J.} {\bf 121}, (2004), 2, 285--341.

\bibitem{6auth} J.-L. Colliot-Th\'el\`ene,
D. Harbater, J. Hartmann, D. Krashen, R. Parimala, V. Suresh,
{ Local-global principles for tori over arithmetic curves},
{\em Algebraic Geom.} {\bf 7} (2020), no. 5,  607--633.

\bibitem{requ} J.-L. Colliot-Th\'el\`ene, J.-J. Sansuc, {La R-\'equivalence
sur les tores}, {\em Ann. Sci. Ec. Norm. Sup.} (4) {\bf 10} (1977),
175--229.

\bibitem{ctsja} \bysame, {
Principal homogeneous spaces under flasque tori: applications},
{\em J. of Algebra} {\bf 106} (1987), 148--205.

\bibitem{endomiy} S. Endo, T. Miyata, {On a classification of the
 function fields of algebraic tori},
{\em Nagoya Math. J.} {\bf 56}, 1974, 85--104.

\bibitem{pcwa} D. Harari, C. Scheiderer, T. Szamuely, {\
Weak approximation for tori over p-adic function fields},
{\em Intern. Math. Res. Notices} {\bf 10} (2015), 2751--2783.

\bibitem{adtcrelle} D. Harari, T. Szamuely,
{Arithmetic duality theorems for 1-motives}, {\em
J. reine angew. Math.}  {\bf 578} (2005), 93--128.

\bibitem{izq} D. Izquierdo,
Principe local-global pour les corps de fonctions sur des corps locaux sup\'erieurs II , {\em  Bull. Soc. Math. France}, 145 (2017), p. 267-293, 2017.

\bibitem{lagarde} L. Lagarde, Degree three unramified cohomology of homogeneous spaces with finite stabilisers, in preparation.

\bibitem{lr} F. Lorenz, P. Roquette,
The theorem of Grunwald-Wang in the setting of valuation theory, in F-W. Kuhlmann et al. (eds.) {\em Valuation theory and its applications}, vol. 2,  Fields Inst. Commun. 33 (2003), 175--212.

\bibitem{milne} J.S. Milne: {\it \'Etale cohomology},
Princeton Mathematical Series {\bf 33}, Princeton University Press,
Princeton, NJ, 1980.

\bibitem{manlinh} M. L. Nguyen, Arithmetics of homogeneous spaces over p-adic function fields, {\em
Journal of the London Mathematical Society} {\bf 109} (1), e12842, 2024.


\bibitem{nis} Y. A. Nisnevich, Espaces homog\`enes principaux rationnellement triviaux et
arithm\'etique des sch\'emas en groupes r\'eductifs sur les anneaux de Dedekind, {\em C. R. Acad.
Sc. Paris} 299(1984) 5--8.

\bibitem{nsw} J. Neukirch, A. Schmidt, K. Wingberg :
{\it Cohomology of number fields}, 2nd ed., Grundlehren der Math. Wiss.
{\bf 323}, Springer-Verlag, 2008.

\bibitem{saltman} D. J. Saltman, Generic Galois extensions and problems in field theory,
{\em Adv. Math.} 43 (1982), 250--283.

\bibitem{corpslocaux} J--P. Serre, {\em Corps locaux}, 3rd ed., Hermann, Paris, 1968.

\bibitem{cogal} \bysame, {\em Cohomologie galoisienne}
(cinqui\`eme
\'edition, r\'evis\'ee et compl\'et\'ee), Springer-Verlag, 1994.

\bibitem{wang}S. Wang,
A counterexample to Grunwald's theorem,
{\em Ann. of Math.} 49 (1948), 1008--1009.

\end{thebibliography}
\end{document}